\documentclass[12pt]{amsproc}

\usepackage{amsthm,amssymb,amsmath,anysize,booktabs,enumerate,hyperref,dsfont}

\hypersetup{citecolor=red, linkcolor=blue, colorlinks=true}

\newtheorem{theorem}{Theorem}[section]
\newtheorem{lemma}[theorem]{Lemma}
\newtheorem{corollary}[theorem]{Corollary}
\newtheorem{conjecture}[theorem]{Conjecture}

\newcommand{\DW}{\mathsf{DW}}
\renewcommand{\DH}{\mathsf{DH}}
\newcommand{\DQ}{\mathsf{DQ}}
\renewcommand{\le}{\leqslant}
\renewcommand{\ge}{\geqslant}
\DeclareMathOperator{\dist}{d}
\newcommand{\allones}{\mathds{1}}

\title{On $m$-ovoids of regular near polygons}

\author[Bamberg]{John Bamberg}
\email{john.bamberg@uwa.edu.au}

\author[Lansdown]{Jesse Lansdown}
\email{jesse.lansdown@research.uwa.edu.au}

\author[Lee]{Melissa Lee}
\email{m.lee16@imperial.ac.uk}

\address[Bamberg, Lansdown]{Centre for the Mathematics of Symmetry of Computation,
School of Mathematics and Statistics, University of Western Australia, Perth WA, Australia.}
\address[Lee]{Department of Mathematics, Imperial College, London SW7 2AZ, UK}

\thanks{The first author acknowledges the support of the Australian Research Council (ARC) 
Future Fellowship FT120100036. 
The second author acknowledges the support of an Australian Postgraduate Award and a UWA Top-Up Scholarship. 
The third author acknowledges the support of a Hackett Postgraduate Research Scholarship.}


\begin{document}

\maketitle

\begin{abstract}
We generalise the work of Segre (1965), Cameron -- Goethals -- Seidel (1978), and Vanhove (2011)
by showing that nontrivial $m$-ovoids of the dual polar spaces 
$\DQ(2d, q)$, $\DW(2d-1,q)$ and $\DH(2d-1,q^2)$ ($d\ge 3$)
are hemisystems. We also provide a more general result that holds for regular near polygons.
\end{abstract}

\section{Introduction}

Near polygons are a large class of point-line incidence geometries that contain the \emph{generalised $2d$-gons}
introduced by J. Tits \cite{Tits:1959aa}, and the dual polar spaces of P. J. Cameron \cite{Cameron:1982aa}.
A \emph{near polygon}, as defined by E. Shult and A. Yanushka \cite{Shult:1980aa}, is a point-line geometry such that for
every point $P$ and line $\ell$, there exists a unique point on $\ell$ nearest to $P$. 
If $d$ is the diameter of the collinearity graph of the near polygon, then we call the near polygon a \emph{near $2d$-gon}. 
A near $2d$-gon is said to be \emph{regular} if its collinearity graph is distance regular. Generalised $2d$-gons are examples 
of near polygons, and the regular near $4$-gons are precisely the finite generalised quadrangles (with an order). However, 
there exist regular near polygons that are not generalised $2d$-gons; for example, every finite dual polar space (of rank at least $3$) 
is an example of a regular near polygon.

An \emph{$m$-ovoid} of a near $2d$-gon is a set of points $\mathcal{O}$ such that every line is incident with exactly $m$ points of $\mathcal{O}$.
The \emph{trivial} $m$-ovoids are the empty set ($m=0$) and the full set of points ($m=s+1$; the number of points on a line).
 For dual polar spaces (that are not generalised quadrangles), the existence of $1$-ovoids is mostly resolved,
however, in rank $3$, it is still not known whether $\mathsf{DQ}^-(7,q)$ or $\mathsf{DH}(6,q^2)$
can contain $1$-ovoids. It follows from \cite[3.4.1]{Payne:2009aa} that there are no $1$-ovoids of $\DW(5,q)$ for $q$ even, and the $q$ odd case was settled by
Thomas \cite[Theorem 3.2]{Thomas:1996aa} (see \cite{CoopersteinPasini}  and \cite[Appendix]{DeBruynPralle} for alternative proofs). De Bruyn and Vanhove reproved this result \cite[Corollary 3.14]{DeBruynVanhoveInequalities} and extended it to other regular near hexagons by showing that a finite generalised hexagon of order $(s, s^3)$ with $s \ge 2$ has no $1$-ovoids \cite[Corollary 3.19]{DeBruynVanhoveInequalities}.

Another interesting case arises in the study of $m$-ovoids when $m$ is exactly half of the number of points on a line. Such an $m$-ovoid is called a \emph{hemisystem}. In 1965, Segre \cite{Segre1965} showed that the only nontrivial $m$-ovoids of $\DH(3, q^2)$, for $q$ odd, are hemisystems. Cameron, Goethals and Seidel \cite{CameronGoethalsSeidel} extended Segre's result to all generalised quadrangles of order $(q, q^2)$, $q$ odd. This was then extended further to regular near $2d$-gons of order $(s,t)$ by Vanhove \cite{Vanhove2011},
which also provided a generalisation of the so-called \emph{Higman bound}: if $s>1$ then the intersection number $c_i$ for all 
$i \in \{1, \ldots, d \}$ obeys the following inequality,
\[
c_i \le \frac{s^{2i}-1}{s^2-1}.
\]
Furthermore, if the bound is sharp for some $c_i$ with $i \in \{2, \ldots, d\}$ then any nontrivial $m$-ovoid is a hemisystem 
\cite[Theorem 3]{Vanhove2011}.

Vanhove showed that for $q$ odd if $\DH(2d-1, q^2)$ has a hemisystem then it induces a distance regular graph with classical parameters \cite[Theorem 4]{Vanhove2011}. Hence the question of the existence of hemisystems in $\DH(2d-1, q^2)$ is of great interest. Now, $\DW(2d -1, q)$ can be embedded in $\DH(2d-1,q)$, and lines in both geometries contain the same number of points. This implies that the intersection of an $m$-ovoid of $\DH(2d-1,q)$ with the points of $\DW(2d-1, q)$ is an $m$-ovoid of $\DW(2d-1,q)$. See also \cite{De-Bruyn:2008aa}.
Therefore, the existence of a hemisystem in $\DH(2d-1, q^2)$ implies the existence of a hemisystem in $\DW(2d-1, q)$, and the existence question can be reframed for $\DW(2d-1, q)$.
In this paper, we extend the work of Segre, Cameron -- Goethals -- Seidel, and Vanhove by showing that the only nontrivial $m$-ovoids of certain dual polar spaces are hemisystems.

\begin{theorem}\label{DQDWDH}
The only nontrivial $m$-ovoids that exist in $\DQ(2d, q)$, $\DW(2d-1,q)$ and $\DH(2d-1,q^2)$, for $d \ge 3$, are hemisystems
(i.e., $m = (q+1)/2$).
\end{theorem}

Theorem \ref{DQDWDH} follows from a more general, but perhaps more technical result, on $m$-ovoids of regular near polygons.
Our main theorem is:

\begin{theorem} \label{ThmMain}
Let $\mathcal{S}$ be a regular near $2d$-gon of order $(s,t_2, t_3, \ldots, t_{d-1}, t)$ satisfying 
\[
t_i+1= \frac{(s^i+(-1)^i)(t_{i-1}+1 +(-1)^is^{i-2})}{s^{i-2}+(-1)^i}
\]
for some $3 \le i \le d$. If a nontrivial $m$-ovoid of $\mathcal{S}$ exists, then it is a hemisystem.
\end{theorem}

De Bruyn and Vanhove {\cite[Theorem 3.2]{DeBruynVanhoveInequalities}} prove that a regular near $2d$-gon (with $s,d\ge 2$) satisfies
\begin{equation}\label{DBVbounds}
\frac{(s^i -1)(t_{i-1} +1 - s^{i-2})}{s^{i-2}-1} \le t_i+1 \le \frac{(s^i +1)(t_{i-1} +1 + s^{i-2})}{s^{i-2} +1}
\end{equation}
for all $i\in\{3,\ldots, d\}$,
and that a finite regular near $2d$-gon with $s\ge2$ and $d\ge3$ which attains the lower bound for $i=3$ 
is isomorphic to $\DQ(2d, s)$, $\DW(2d-1, s)$ or $\DH(2d-1, s^2)$, where $s$ is a prime power \cite[Theorem 3.5]{DeBruynVanhoveInequalities}.
Note that the hypothesis of Theorem \ref{ThmMain} is valid when the the upper bound is met for $i$ even, or when the lower bound
is met for $i$ odd, in the De Bruyn--Vanhove bounds \eqref{DBVbounds}. Theorem \ref{DQDWDH} follows directly from 
\cite[Theorem 3.5]{DeBruynVanhoveInequalities} and Theorem \ref{ThmMain}.

\section{Background}
This section contains information on some of the key facts about regular near $2d$-gons and $m$-ovoids, which will be useful later in the paper. For greater depth, we refer the reader to Brouwer, Cohen and Neumaier's book \cite{BCN}.

Let $\Gamma$ be a connected, undirected graph without loops. The \emph{distance} between two vertices $x$ and $y$, denoted $\dist(x,y)$, is the shortest path length from $x$ to $y$, and the maximum distance between any two given points is the \emph{diameter} $d$ of $\Gamma$. The set of all vertices at distance $i$ from $x$ is denoted by $\Gamma_i(x)$. A graph $\Gamma$ of diameter $d$ is said to be \emph{distance regular} if there exist numbers $b_i$ for $i \in \{0, \ldots, d-1 \}$ and $c_i$ for $i \in \{1 , \ldots, d \}$ such that $b_i = |\Gamma_{i+1} (x) \cap \Gamma_1(y)|$ and $c_i = |\Gamma_{i-1} (x) \cap \Gamma_1(y)|$ for all $x$ and $y$ at distance $i$ in $\Gamma$. If a graph is distance regular then there also exist constants $a_i = |\Gamma_i(x) \cap \Gamma_1(y)|$ for all $x$ and $y$ at distance $i$ with $i \in \{ 1, \ldots, d-1 \}$. We call $a_i$, $b_i$ and $c_i$ the \emph{intersection numbers} of $\Gamma$. 

Given $x$ and $y$ at distance $l$, there are $p_{i, j}^l$ vertices that are at distance $i$ from $x$ and distance $j$ from $y$. Furthermore, $p_{1,i}^{i-1} = b_{i-1}$, $p_{1, i}^i = a_i$, $p_{1, i}^{i+1} = c_{i+1}$ and $p_{i,j}^l = p_{j, i}^l$. Therefore, combining \cite[Lemma 4.1.7]{BCN} and \cite[\S 4.1 (10)]{BCN}, we may calculate $p_{i+1, j}^l$ recursively using the following formula.
\[
p_{i+1,j}^l = \frac{ p_{i,j}^{l-1} c_l +  p_{i,j}^l a_l +p_{i,j}^{l+1} b_l - p_{i-1,j}^l b_{i-1} -  p_{i,j}^l a_i}{c_{i+1}}.
\]
We also define the $i$-distance \emph{valencies} of the graph, $k_i := p_{i, i}^ 0$ for $i \in \{0, 1, \ldots, d \}$ (and so $k_1=s(t+1)$). 

Given a graph $\Gamma$ of diameter $d$, for any distance $i$, the \emph{adjacency matrix} $A_i$ is the matrix indexed by the vertices of $\Gamma$, with entries
\[ (A_{i})_{xy} = 
	\begin{cases} 
      1 & \text{if }\dist(x, y)=i \\
      0 & otherwise.
      \end{cases} \]
The set of adjacency matrices $\{ A_0, A_1 , \ldots, A_d \}$ forms a basis for the \emph{Bose--Mesner algebra} for $\Gamma$ which also has a unique basis of minimal idempotents $\{E_0, E_1, \ldots , E_d \}$ (see \cite[\S2.6]{BCN}). As a result, the Bose--Mesner algebra can be decomposed into mutually orthogonal subspaces corresponding to the image of each minimal idempotent. By convention, $E_0$ has rank $1$, that is, $E_0=\frac{1}{n} J$ where $J$ is the `all ones' matrix and $n$ is the number of vertices of $\Gamma$. The \emph{dual degree set} of a vector $v$ is the set of indices of the minimal idempotents $E_i$ such that $v E_i \not = 0$ and $i \not = 0$. Two vectors are called \emph{design-orthogonal} when their dual degree sets are disjoint. The following lemma about design-orthogonal vectors will be useful in the proof of the main theorem. It can be found in \cite[Theorem 6.7]{Delsarte:1977aa}, and is given here with a proof for completeness.

\begin{lemma} \label{design-orthog}
If $f$ and $g$ are design-orthogonal vectors, then $f \cdot g = \frac{(f \cdot \allones)(g \cdot \allones)}{n}$, where $\allones$ is the `all-ones' vector.
\end{lemma}

\begin{proof}
Let $\alpha = \frac{f \cdot \allones}{\allones \cdot \allones}$ and $\beta = \frac{g \cdot \allones}{\allones \cdot \allones}$. So $(f - \alpha \allones) \cdot \allones = 0$ and $(g - \beta \allones) \cdot \allones = 0$. 
Since $f$ and $g$ are design-orthogonal, $(f - \alpha \allones)$ and $(g - \beta \allones)$ belong to a pair of direct sums of eigenspaces that intersect trivially and hence
$( f - \alpha \allones) \cdot (g - \beta \allones) = 0$. Thus 
\begin{align*}
f \cdot g &= \alpha (g \cdot \allones) + \beta (f \cdot \allones) - \alpha \beta \allones \cdot \allones \\
&= \frac{(f \cdot \allones) (g \cdot \allones)}{ \allones \cdot \allones} + \frac{(f \cdot \allones)(g \cdot \allones)}{\allones \cdot \allones} - \frac{(f \cdot \allones) (g \cdot \allones)}{ \allones \cdot \allones} \\
&= \frac{(f \cdot \allones)(g \cdot \allones)}{n}.\qedhere
\end{align*}
\end{proof}

A \emph{near polygon}, or near $2d$-gon ($d \ge 2$) is an incidence geometry such that

\begin{enumerate}
	\item every two points lie on at most one line,
	\item any two points are at most at distance $d$ in the collinearity graph, and
	\item given a line $\ell$ and a point $P$ there is a unique point $Q$ on $\ell$ which is nearest to $P$ with respect to distance in the collinearity graph.
\end{enumerate}

A near polygon that has $t+1$ lines on each point and $s+1$ points on each line is said to have \emph{order} $(s,t)$. If in a near polygon of order $(s,t)$ there also exist constants $t_i$ for $i \in \{0, \ldots, d \}$ such that there are $t_i +1$ lines on $y$ containing a point at distance $i-1$ from $x$ whenever two points $x$ and $y$ are at distance $i$, then such a near polygon is called \emph{regular}, with parameters $(s, t_2, t_3, \ldots, t_{d-1}, t)$. Examples of regular near $2d$-gons include the finite dual polar spaces; the point-line geometries obtained by taking
the maximal totally isotropic subspaces of a finite polar space for the points, and the next-to-maximal subspaces
for the lines. We refer the reader to \cite[\S1.9.5]{NearPolygons} for more on the definition of 
a dual polar space. In this paper, we will only be concerned with 
$\DW(2d-1, s)$, $\DQ(2d, s)$, and $\DH(2d-1, s^2)$.

The finite regular near polygons are exactly the near polygons with distance regular collinearity graphs. Moreover, for all $i \in \{0, \ldots, d\}$
\[
a_i = (s-1)(t_i +1),\quad b_i = s(t - t_i),\quad c_i = t_i +1.
\]
By definition, $t_0=-1$ and $t_1=0$.
In a regular near $2d$-gon with parameters $(s, t_2, t_3, \ldots, t_{d-1}, d)$, we have the following relations.

\begin{lemma}\label{intersectionNumbers} \cite[\S4.1 (7); (9); (1c)]{BCN} 
\begin{align*} 
k_i&= k_{i-1}\frac{b_{i-1}}{c_i} =  s k_{i-1} \frac{t-t_{i-1}}{t_i+1}\quad (1\le i\le d),\\
p_{i,j}^l k_l &= p_{l,j}^i k_i\quad (0\le i,j,\ell, d),\\ 
p_{1,i}^{i+1} &= c_{i+1} \quad (1\le i\le d).  
\end{align*}
\end{lemma}

Lemma \ref{intersectionNumbers} gives the following corollary.

\begin{corollary} \label{pii-1}
Let $1\le i\le d$. Then
\begin{align*}
p_{i,i-1}^1 = \frac{k_i c_i}{k_1} = \frac{k_i (t_i +1)}{s(t+1)}.
\end{align*}
\end{corollary}

The following lemma follows directly from the definition of an $m$-ovoid and the fact that there 
are $s+1$ points on every line of a finite regular near $2d$-gon $\mathcal{S}$ with parameters $(s,t_2,\ldots, t_{d-1},t)$.

\begin{lemma}\label{complement}
The complement of an $m$-ovoid of $\mathcal{S}$ is a $(s+1-m)$-ovoid.
\end{lemma}

\begin{lemma}[{\cite[Lemma 5]{Vanhove2011}}] \label{Vanhove5}
If $\mathcal{O}$ is an $m$-ovoid of $\mathcal{S}$, then for every $i \in \{0, 1, \ldots, d \}$ and $x\in\mathcal{O}$, 
\[
|\Gamma_i (x) \cap \mathcal{O}| = k_i \bigg( \frac{m}{s+1} + \bigg(- \frac{1}{s}\bigg)^i \bigg(1 - \frac{m}{s+1}\bigg) \bigg).
\]
\end{lemma}

By Lemmas \ref{complement} and \ref{Vanhove5}, we have the following:

\begin{corollary} \label{corollaryOfVanhove}
If $\mathcal{O}$ is an $m$-ovoid of $\mathcal{S}$, then for every $i \in \{0, 1, \ldots, d \}$ and $x\notin\mathcal{O}$, 
\[
|\Gamma_i (x) \cap \mathcal{O}| =  k_i \frac{m}{s+1} \left(1 - \left(\frac{-1}{s}\right)^i\right).
\]
\end{corollary}

\section{Proof of the main result}

We now prove Theorem \ref{ThmMain}. Recall that we are assuming that
\[
c_i=t_i+1= \frac{(s^i+(-1)^i)(c_{i-1} +(-1)^is^{i-2})}{s^{i-2}+(-1)^i}
\]
for some $3\le i\le d$.

\begin{proof}
Let $\mathcal{O}$ be a nontrivial $m$-ovoid of $\mathcal{S}$. 
Throughout this proof, we will let $\chi_\mathcal{O}$ denote the \emph{characteristic vector} of $\mathcal{O}$ with respect to the set
of points $\mathcal{P}$:
\[ (\chi_\mathcal{O})_y = 
	\begin{cases} 
      1 & \text{if }y \in \mathcal{O} \\
      0 & otherwise.
      \end{cases} \]

A simple double counting argument shows that $|\mathcal{O}|$ is equal to 
$m|\mathcal{L}|/(t+1)$
where $\mathcal{L}$ is the set of lines of $\mathcal{S}$. If we also count flags (i.e., point-line incident pairs), then
$|\mathcal{P}|(t+1)=|\mathcal{L}|(s+1)$ where $\mathcal{P}$ is the set of points of $\mathcal{S}$,
and hence
\[
|\mathcal{O}|=\frac{mn}{s+1},
\]
where $n=|\mathcal{P}|$.

Recall that $3\le i\le d$.
Now we fix an element $x\notin \mathcal{O}$ and count pairs $(y,z)$ of elements of $\mathcal{O}$ such that $\dist(x,y) =i$ and either
$\dist(y,z)=i-1$ and $\dist(x,z)=1$, or $\dist(y,z)=1$ and $\dist(x,z)=i-1$.

Let $x$ and $y$ be two points at distance $i$. Define $v_{x,y}$ as in \cite[Theorem 3.2(b)]{DeBruynVanhoveInequalities},
\[
v_{x,y}:=s(c_{i-1}+(-1)^is^{i-2})(\chi_{x}+\chi_{y})+\chi_{\Gamma_1(x)\cap \Gamma_{i-1}(y)}+
\chi_{\Gamma_{i-1}(x)\cap \Gamma_1(y)}.
\]
Note that
\[
v_{x,y}\cdot \allones = 2( s(c_{i-1}+(-1)^is^{i-2})+p_{1,i-1}^i ) =2( s(c_{i-1}+(-1)^is^{i-2})+c_i )
\]
and furthermore that $v_{x,y}$ and $\chi_{\mathcal{O}}$ are design-orthogonal \cite[Theorem 3.2]{DeBruynVanhoveInequalities} and hence by Lemma \ref{design-orthog}, 
\[
\mu:=v_{x,y}\cdot\chi_\mathcal{O} =2( s(c_{i-1}+(-1)^is^{i-2})+c_i )m/(s+1).
\]
Let $\Gamma$ be the collinearity graph of $\mathcal{S}$.

Counting first $y$ and then $z$, the number of pairs is
\begin{align*}
& \sum_{y\in \mathcal{O}\cap \Gamma_i(x)} \left( | \Gamma_1(x) \cap \Gamma_{i-1}(y) \cap \mathcal{O}| + | \Gamma_{i-1}(x) \cap \Gamma_1(y) \cap \mathcal{O} | \right)\\
=&\sum_{y\in \mathcal{O}\cap \Gamma_i(x)} (v_{x,y}-s(c_{i-1}+(-1)^is^{i-2})(\chi_{x}+\chi_{y}))\cdot\chi_\mathcal{O}\\
=&|\mathcal{O}\cap \Gamma_i(x)| (\mu-s(c_{i-1}+(-1)^is^{i-2}))\\
=&|\mathcal{O}\cap \Gamma_i(x)| \left(\frac{2\left( s(c_{i-1}+(-1)^is^{i-2})+c_i \right)m}{s+1} -s(c_{i-1}+(-1)^is^{i-2})\right)\\
=&|\mathcal{O}\cap\Gamma_i(x)| \left(\frac{2 c_i m s-(s+1-2 m) \left(c_{i-1} s^2+(-1)^i s^i\right)}{s (s+1)}\right).
\end{align*}
Now, by Corollary \ref{corollaryOfVanhove} and Lemma \ref{intersectionNumbers}, \begin{align*}
|\mathcal{O}\cap \Gamma_i(x)| c_i &= k_i c_i \frac{m}{s+1}\left(1 - \left(- \frac{1}{s}\right)^i\right) =sk_{i-1}(t-t_{i-1})\frac{m}{s+1}\left(1 - \left(- \frac{1}{s}\right)^i\right) 
\end{align*}
and hence the number of pairs $(y,z)$ is
\small \begin{equation}\label{firstcount}
 \frac{mk_{i-1} (t - t_{i-1})}{s+1}\left(1 - \left(- \frac{1}{s}\right)^i\right)
\frac{  2 c_i m s-(s+1-2 m) \left(c_{i-1} s^2+(-1)^i s^i\right)}{c_i (s+1)}.
\end{equation}\normalsize

Now we consider the pairs the opposite way, namely counting $z$ then $y$. The number of pairs $(z,y)$ is equal to
\[
\sum_{z\in \mathcal{O}\cap \Gamma_1(x)} | \Gamma_{i-1}(z) \cap \Gamma_i(x) \cap\mathcal{O}|
+ \sum_{z\in \mathcal{O}\cap \Gamma_{i-1}(x)} | \Gamma_1(z) \cap \Gamma_i(x) \cap\mathcal{O}|.
\]
Suppose $\dist(z,x)=i-1$. There are $t+1$ lines on $z$, and the set of points incident with these lines, other than the point $z$ itself, form $\Gamma_1(z)$. There are $t_{i-1}+1$ lines on $z$ incident with a unique point at distance $i-2$ from $x$, and the remaining points on these lines are at distance $i-1$ from $x$. Moreover, $z$ is the unique nearest point to $x$ with distance $i-1$ for the remaining $t - t_{i-1}$ lines, and hence any other point on these lines must have distance $i$ from $x$. Since $z$ is in $\mathcal{O}$, there are $m-1$ additional points of $\mathcal{O}$ on each such line.
Therefore,  
\[
| \Gamma_1(z) \cap \Gamma_i(x) \cap\mathcal{O}|=    (t-t_{i-1})(m-1).
\]

Now suppose $\dist(z,x)=1$. We will compute $ |\Gamma_{i-1}(z) \cap \Gamma_i(x) \cap\mathcal{O}|$.
Take $z \in \mathcal{O} \cap \Gamma_1 (x)$ and consider a point $w \in \Gamma_{i-2}(z) \cap \Gamma_{i-1}(x)$. Note that any point $y$ is collinear with some such point $w$, giving rise to the following equation:
{\small
\begin{align}\label{doublecount}
 \sum_{y \in \Gamma_{i-1}(z) \cap \Gamma_i(x) \cap \mathcal{O}}|\Gamma_{i-2}(z) \cap \Gamma_{i-1}(x) \cap \Gamma_1 (y)|
=& \sum_{w \in \Gamma_{i-2}(z) \cap \Gamma_{i-1}(x) \cap \mathcal{O}}\hspace{-1cm} |\Gamma_{i-1}(z) \cap \Gamma_{i}(x) \cap \mathcal{O} \cap \Gamma_1 (w)|\\ 
+& \sum_{w \in \Gamma_{i-2}(z) \cap \Gamma_{i-1}(x) \cap \mathcal{O}^c}\hspace{-1cm}  |\Gamma_{i-1}(z) \cap \Gamma_{i}(x) \cap \mathcal{O} \cap \Gamma_1 (w)|\notag
\end{align}}
where $\mathcal{O}^c$ is the complement of $\mathcal{O}$ within the set of points of $\mathcal{S}$.

 Let $\ell$ be a line through $w$. There is a point on $\ell$ which is the unique closest point to $x$. If this point is $w$, then every other point must be at distance $i$ from $x$. If this point is not $w$, then it must be distance $i - 2$ from $x$, and every other point on $\ell$ is distance $i -1$ from $x$. There are $t_{i-1} +1$ lines on $w$ with a unique point at distance $i - 2$ from $x$. Hence there are $t - t_{i-1}$ lines $\ell'$ for which $w$ is the unique nearest point to $x$ and every other point on $\ell'$ is at distance $i$ from $x$. Moreover, note that if a point $y$ is at distance $i$ from $x$, then it cannot be distance $i-2$ from $z$, since $\dist(x,z) = 1$, and thus any point other than $w$ on any line $\ell'$ is in $\Gamma_{i-1}(z) \cap \Gamma_{i}(x) \cap \Gamma_1 (w)$. There are $m-1$ such points in $\mathcal{O}$ when $w \in \mathcal{O}$, otherwise there are $m$ such points in $\mathcal{O}$.

There are $t_{i-1}$ lines on any point $y$ which have a unique point at distance $i-2$ from $z$, and hence also at distance $i-1$ from $x$. Recalling that $|\Gamma_{i-2}(z) \cap \Gamma_{i-1}(x)| = p_{i-1,i-2}^1$, our Equation \eqref{doublecount} becomes:
\begin{align*}
|\Gamma_{i-1}(z) \cap \Gamma_i(x) \cap \mathcal{O}|(t_{i-1} +1) & = |\Gamma_{i-2}(z) \cap \Gamma_{i-1}(x) \cap \mathcal{O}| (t - t_{i-1})(m-1) \\
& \quad \quad \quad + |\Gamma_{i-2}(z) \cap \Gamma_{i-1}(x) \cap \mathcal{O}^c| (t - t_{i-1})m \\
& = |\Gamma_{i-2}(z) \cap \Gamma_{i-1}(x) \cap \mathcal{O}| (t - t_{i-1})(m-1)\\
& \quad \quad \quad + (p_{i-1,i-2}^1 - |\Gamma_{i-2}(z) \cap \Gamma_{i-1}(x) \cap \mathcal{O}|)(t - t_{i-1})m \\
& = p_{i-1,i-2}^1 (t - t_{i-1})m - |\Gamma_{i-2}(z) \cap \Gamma_{i-1}(x) \cap \mathcal{O}| (t - t_{i-1}).
\end{align*}
Hence we obtain an iterative formula,
\[
|\Gamma_{i-1}(z) \cap \Gamma_{i}(x) \cap \mathcal{O}| = p_{i-1,i-2}^1 \frac{t-t_{i-1}}{t_{i-1}+1}m - \frac{t - t_{i-1}}{t_{i-1}+1} |\Gamma_{i-2}(z) \cap \Gamma_{i-1}(x) \cap \mathcal{O}|,
\]
which, with the help of Lemma \ref{intersectionNumbers} and Corollary \ref{pii-1}, we can write as a recurrence relation
\[
s f_i = m-f_{i-1},\quad f_1=1
\]
where $f_i:=\frac{1}{p_{i,i-1}^1} |\Gamma_{i-1}(z) \cap \Gamma_{i}(x) \cap \mathcal{O}|$ for all $i\ge 1$.
(Note: $|\Gamma_{0}(z) \cap \Gamma_{1}(x) \cap \mathcal{O}| = 1$ and $p_{1,0}^1=1$).
Therefore, by the elementary theory of recurrence relations, we have
\[
f_i=\frac{m-s \left(-\frac{1}{s}\right)^i (-m+s+1)}{s+1}
\]
for all $i\ge 1$. Hence, by Corollary \ref{pii-1},
\begin{align*}
|\Gamma_{i-1}(z) \cap \Gamma_{i}(x) \cap \mathcal{O}|&=p_{i,i-1}^1\left(\frac{m-s \left(-\frac{1}{s}\right)^i (-m+s+1)}{s+1}
\right)\\
&= \frac{k_{i-1} (t - t_{i-1})}{t+1}\left(\frac{m-s \left(-\frac{1}{s}\right)^i (-m+s+1)}{s+1}
\right)\\
&=  \frac{k_{i-1} (t - t_{i-1})}{s^{i-1}(t+1)} \left(\frac{m}{s+1} \left(s^{i-1} + (-1)^{i-2}\right) + (-1)^{i-1} \right).
\end{align*}

Now, making use of Corollary \ref{corollaryOfVanhove}, we sum our two terms together:
\begin{align*} &\sum_{z \in \mathcal{O} \cap \Gamma_1(x)} |\Gamma_{i-1}(z) \cap \Gamma_i (x) \cap \mathcal{O}|
+ \sum_{z \in \mathcal{O} \cap \Gamma_{i-1}(x)} |\Gamma_{1}(z) \cap \Gamma_i (x) \cap \mathcal{O}| \\
=&|\mathcal{O} \cap \Gamma_1(x)||\Gamma_{i-1}(z) \cap \Gamma_i (x) \cap \mathcal{O}| + |\mathcal{O} \cap \Gamma_{i-1}(x)||\Gamma_{1}(z) \cap \Gamma_i (x) \cap \mathcal{O}|\\
=&s(t+1) \frac{m}{s+1}\left(1 + \frac{1}{s}\right)\frac{k_{i-1} (t - t_{i-1})}{s^{i-1}(t+1)} \left(\frac{m}{s+1} \left(s^{i-1} + (-1)^{i-2}\right) + (-1)^{i-1} \right) \\
&\quad \quad \quad \quad + k_{i-1} \frac{m}{s+1}\bigg(1 - \left(- \frac{1}{s}\right)^{i-1}\bigg)(t-t_{i-1})(m-1) \\
=&\frac{mk_{i-1} (t - t_{i-1})}{s+1}\left(   m\left(1-\left(-\frac{1}{s}\right)^{i-1} \right) +
\left(\frac{-1}{s}\right)^{i-1}(s+1)+ (m-1)\left(1 - \left(- \frac{1}{s}\right)^{i-1}\right) \right)
\end{align*}
and therefore, the number of pairs $(z,y)$ is
\begin{equation}\label{secondcount}
\frac{mk_{i-1} (t - t_{i-1})}{s+1}\left(   2m-1 +\left(\frac{-1}{s}\right)^{i-1}(s-2m+2) \right).
\end{equation}

Equating the two counts, \eqref{firstcount} and \eqref{secondcount} yields
\begin{align*}
& \left(1 - \left(- \frac{1}{s}\right)^i\right)\frac{  2 c_i m s-(s+1-2 m) \left(c_{i-1} s^2+(-1)^i s^i\right)}{c_i (s+1)}\\
&= 2m-1 +\left(\frac{-1}{s}\right)^{i-1}(s-2m+2).
\end{align*}
Taking the difference of each side of the above equation and factoring gives
\[
\frac{s^{-i} (s+1-2 m) \left(c_i \left(s^i+(-1)^i (s+2) s\right)+\left((-1)^i-s^i\right) \left(c_{i-1} s^2+(-1)^i s^i\right)\right)}{c_i (s+1)}=0
\]
and hence 
\begin{equation}\label{equalto0}
(s+1-2 m) \left(c_i \left(s^i+(-1)^i (s+2) s\right)+\left((-1)^i-s^i\right) \left(c_{i-1} s^2+(-1)^i s^i\right)\right)=0
\end{equation}
Now by assumption,
\[
c_{i-1}s^2 +(-1)^is^i
=c_i\frac{s^i+(-1)^is^2}{s^i+(-1)^i}
\]
and hence
\begin{align*}
&c_i \left(s^i+(-1)^i (s+2) s\right)+\left((-1)^i-s^i\right) \left(c_{i-1} s^2+(-1)^i s^i\right)\\
&=c_i \left(s^i+(-1)^i (s+2) s +\left((-1)^i-s^i\right) \frac{s^i+(-1)^is^2}{s^i+(-1)^i} \right)\\
&=c_i \frac{2 (-1)^i (s+1) \left(s^i+(-1)^i s\right)}{s^i+(-1)^i}.
\end{align*}
Since $i>1$, we have $s^i+(-1)^i s\ne 0$, and therefore,
Equation \eqref{equalto0} becomes $m=(s+1)/2$.
\end{proof}

\section{Further results and computation}

Theorem \ref{DQDWDH} leaves open the natural question of whether there exist hemisystems of $\DQ(6, q)$, $\DW(5,q)$ and $\DH(5,q^2)$.
Firstly, De Bruyn and Vanhove announced in conference presentations
that there are no hemisystems of $\DW(5,3)$, and that there is a unique example
for $\DQ(6,3)$. We thank the referee and Michel Lavrauw for mentioning these results to us.
For small values of (odd) $q$, we have found examples for $\DQ(6,q)$, and we have listed the known examples in Table \ref{table:DQ}. In particular, we
could show by using the computer algebra system \textsc{GAP} \cite{GAP4}, a package \textsc{FinInG} \cite{fining}, and the mixed-integer programming software \textsc{Gurobi} \cite{gurobi} that there
is a unique example up to equivalence in $\DQ(6,3)$. For $\DQ(6,5)$, there were numerous examples found 
admitting an element of order $5$ or $9$, but we were unable to enumerate them all.

\begin{table}[ht]
\begin{tabular}{ccc}
\toprule
$q$ & Stabiliser  & Number up to equivalence\\
\midrule
$3$ & $2\times A_5$ & 1 \\
$5$ & $D_{60}$ & $4$\\ 
& $D_{20}$ & $16$\\
\bottomrule
\end{tabular}
\caption{Some known examples of hemisystems of $\DQ(6,q)$, for small $q$.}\label{table:DQ}
\end{table}
For $\DW(5,q)$, it seems the situation is different, despite its combinatorial parameters being identical to those of $\DQ(6,q)$.
By computer, we showed that there are no hemisystems of $\DW(5,q)$ for $q\in\{3,5\}$. We make the following conjectures:

\begin{conjecture}
There are no hemisystems of $\DW(5,q)$, for all prime powers $q$.
\end{conjecture}

If true, this would also imply that there are no hemisystems of $\DH(5,q^2)$, for all prime powers $q$, answering a problem posed by Vanhove
\cite[Appendix B, Problem 7]{VanhoveThesis}.

\begin{conjecture}
For each odd prime power $q$, there exists a hemisystem of $\DQ(6,q)$.
\end{conjecture}


\end{document}